\documentclass[final,nomarks]{dmtcs-episciences}
\usepackage[utf8]{inputenc}
\usepackage{amsmath}
\usepackage{amssymb}
\usepackage{amsthm}
\usepackage{graphicx}
\usepackage{hyperref}
\usepackage{color}
\usepackage{numprint}
\npthousandsep{\ensuremath\,\allowbreak}

\usepackage[round]{natbib}

\newtheorem{thm}{Theorem}[section]
\newtheorem{prop}[thm]{Proposition}
\newtheorem{lem}[thm]{Lemma}
\newtheorem{coro}[thm]{Corollary}
\newtheorem{open}{Open question}
\theoremstyle{remark}
\newtheorem{rmk}[thm]{Remark}

\newcommand{\tdef}[1]{\textcolor{blue}{\emph{#1}}}

\newcommand{\sym}{\mathfrak{S}}
\newcommand{\cyc}[1]{\overline{#1}}
\newcommand{\pin}{\operatorname{Pin}}
\newcommand{\cycf}{\operatorname{cyc}}
\newcommand{\dyck}{\mathcal{D}}
\newcommand{\dyckc}{\operatorname{dc}}
\newcommand{\dycktype}{\mathfrak{D}}
\newcommand{\motz}{\mathcal{M}}
\newcommand{\motzf}{\mathfrak{M}}
\newcommand{\meand}{\mathcal{R}}
\newcommand{\ord}{\operatorname{ord}}
\newcommand{\ordset}{\mathcal{O}}

\title{Efficient recurrence for the enumeration of permutations with fixed pinnacle set}

\author{Wenjie Fang}
\affiliation{LIGM, Université Gustave Eiffel, CNRS, ESIEE Paris, 77454 Marne-la-Vallée, France}
\keywords{permutation, pinnacle set, enumeration, recurrence}
\received{2021-07-30}

\revised{2022-02-15}

\accepted{2022-02-15}

\publicationdetails{24}{2022}{1}{8}{8321}

\begin{document}

\maketitle

\begin{abstract}
Initiated by Davis, Nelson, Petersen and Tenner (2018), the enumerative study of pinnacle sets of permutations has attracted a fair amount of attention recently. In this article, we provide a recurrence that can be used to compute efficiently the number $|\sym_n(P)|$ of permutations of size $n$ with a given pinnacle set $P$, with arithmetic complexity $O(k^4 + k\log n)$ for $P$ of size $k$. A symbolic expression can also be computed in this way for pinnacle sets of fixed size. A weighted sum $q_n(P)$ of $|\sym_n(P)|$ proposed in Davis, Nelson, Petersen and Tenner (2018) seems to have a simple form, and a conjectural form is given recently by Flaque, Novelli and Thibon (2021+). We settle the problem by providing and proving an alternative form of $q_n(P)$, which has a strong combinatorial flavor. We also study admissible orderings of a given pinnacle set, first considered by Rusu (2020) and characterized by Rusu and Tenner (2021), and we give an efficient algorithm for their counting.
\end{abstract}

\section{Introduction}

Given a permutation $\pi \in \sym_n$ in one-line notation $\pi_1 \pi_2 \cdots \pi_n$, we consider its local maxima, \textit{i.e.}, elements $\pi_i$ for $2 \leq i \leq n-1$ such that $\pi_{i-1} < \pi_i > \pi_{i+1}$. In this case, the index $i$ is called a \tdef{peak} of $\pi$, and the element $\pi_i$ is called a \tdef{pinnacle}. We denote by $[n]$ the set $\{1, 2, \ldots, n\}$. The \tdef{pinnacle set} of a permutation $\pi$, denoted by $\pin(\pi)$, is the set of pinnacles of $\pi$. For $P \subseteq [n]$, we denote by $\sym_{n}(P)$ the set of permutations with pinnacle set $P$. Peaks of permutations have already been much studied, partly due to its link to the algebraic aspect of the symmetric group through the peak algebra. Inspired by the studies on peaks, pinnacle sets of permutations were first explored in \cite{pin-2018}, and various results were given there.

One of the main questions asked in \cite{pin-2018} is how to compute $|\sym_n(P)|$ efficiently. Two recurrences were proposed there, but both need super-exponential time in the number of pinnacles to compute. In \cite{pin-2021}, an improved formula was given, but it is still a sum of super-exponentially many terms. A formula with only exponentially many terms was given recently in \cite{pin-sagan} using the inclusion-exclusion principle. Then, a new and simpler recurrence was proposed in \cite{pin-marne} that drastically improves the arithmetic complexity, \textit{i.e.}, the number of arithmetic operations, to $O(\max(P)|P|^2)$, but there is still a dependency on the values in $P$, thus not polynomial in the bit-length of input. In this article, we again improve on the arithmetic complexity, giving a recurrence (see Theorem~\ref{thm:weight-rec}) that can be used to compute $|\sym_n(P)|$ with polynomial arithmetic complexity in the bit-length of input.

\begin{prop}\label{prop:algo}
  There is an algorithm that computes $|\sym_n(P)|$ with $O(|P|^2 \log n + |P|^4)$ arithmetic operations.
\end{prop}

Using the same recurrence in Theorem~\ref{thm:weight-rec}, through a computer algebra system, we also obtain general symbolic expressions of $|\sym_n(P)|$ for $P$ of arbitrary fixed size, extending formulas given in \cite{pin-2018} for $|P|=1,2$. However, these expressions are messy in general. In \cite{pin-2018}, it was suggested that the following weighted sum may have a simpler form:
\begin{equation}
  \label{eq:q-def}
  q_n(P) = \sum_{Q \subseteq P} 2^{|Q|} |\sym_n(Q)|.
\end{equation}
It is also observed in \cite{pin-2018} that we may use $q_n(P)$ to recover $|\sym_n(P)|$ with the inclusion-exclusion principle. In \cite{pin-marne}, a general process of generating an expression of $q_n(P)$ was conjectured, with a few examples of such expressions for small $|P|$, which are indeed simpler. We settle this problem by providing an alternative form of $q_n(P)$ as a summation over some combinatorial objects. Our formula is equivalent to the one conjectured in \cite{pin-marne} (see Remark~\ref{rmk:q-conj}).

For $s \leq k$, let $\meand_{k,s}$ be the set of sequences $r = (r_0, r_1, \ldots, r_k)$ with the conditions
\begin{itemize}
\item $r_0 = 0$, $r_k = s$, $r_i \geq 0$ for all $0 \leq i \leq k$;
\item $r_i = r_{i-1} \pm 1$ for all $1 \leq i \leq k$.
\end{itemize}
We can interpret $\meand_{k,s}$ as the set of $y$-coordinate sequences of Dyck meanders of length $k$ terminating on $y = s$. We define $\meand_k = \cup_{s = 0}^k \meand_{k,s}$.

\begin{thm}[See also Conjecture~5.1 in \cite{pin-marne}]\label{thm:q}
  For $n \geq 1$ and $P = \{p_1 > p_2 > \cdots > p_k\} \subseteq [n]$, we take the convention that $p_0 = n+1$ and $p_{k+1} = 1$. Given $r \in \mathcal{R}_{k}$, we define its weight $w_\meand(r)$ by
  \[
    w_\meand(r) = \prod_{m=0}^k (r_m + 1)^{p_m-p_{m+1}}.
  \]
  Then we have
  \[
    q_n(P) = 2^{n-k-1} \sum_{r \in \mathcal{R}_k} w_\meand(r).
  \]
\end{thm}

Although the statement of Theorem~\ref{thm:q} is purely combinatorial, our proof uses heavy computations. We thus ask naturally for a more satisfying combinatorial proof.

For permutations with the same pinnacle set, it is possible that their pinnacles appear in different orders, and not all orders are possible. A pinnacle order is called \emph{admissible} if there is a permutation in which these pinnacles appear in that order. Given a pinnacle set $P$, we denote by $\ordset(P)$ the set of all admissible pinnacle orders of $P$ (its formal definition is postponed to Section~\ref{sec:pin-order}). Admissible pinnacle order was first considered in \cite[Question~3]{rusu-sorting}. A characterization of $\ordset(P)$ was given in Theorem~3.6 of \cite{rusu-tenner} using the language of interruptions. In the same paper, the authors asked for a function that computes $|\ordset(P)|$. As a response, a formula of $|\ordset(P)|$ was given in \cite{pin-sagan}, but it involves exponentially many terms in the size of $P$. In the article, we provide an efficient way to compute $|\ordset(P)|$ using a recurrence.

\begin{prop}\label{prop:pin-order-algo}
  There is an algorithm that computes $|\ordset(P)|$ with $O(|P|^2)$ arithmetic operations.
\end{prop}

This article is organized as follows. Section~\ref{sec:motz} reduces the counting of $\sym_n(P)$ to that of some weighted Motzkin paths, using a variant of the Françon-Viennot bijection (Proposition~\ref{prop:sum-type}). It is then further reduced in Section~\ref{sec:dyck} to the counting of appropriately weighted Dyck paths by compressing horizontal steps in the weighted Motzkin paths (Theorem~\ref{thm:dyck-weight}). Given the formulation in Dyck paths, we provide in Section~\ref{sec:rec} the main result of this article (Theorem~\ref{thm:weight-rec}), a recurrence that allows us to compute $|\sym_n(P)|$ efficiently, as in Proposition~\ref{prop:algo}. Using a variant of this recurrence, we provide in Section~\ref{sec:q} a proof of Theorem~\ref{thm:q}. We then deal with the enumeration of admissible pinnacle orderings in Section~\ref{sec:pin-order}.

\paragraph{Acknowledgment} We would like to thank Jean-Christophe Novelli for bringing this subject to the attention of the author, for interesting discussions and for advice on the draft of this article. We would also like to thank \'Eric Fusy for pointing out the link between the construction in Section~\ref{sec:motz} and the Françon-Viennot bijection in \cite{francon-viennot, flajolet}. This work is not supported by any funding with precise predefined goal, but it is supported by the publicly funded laboratory LIGM of Université Gustave Eiffel.

\section{Construction of permutations with a fixed pinnacle set} \label{sec:motz}

In the following, we consider \tdef{cyclic permutations}, which are equivalent classes of permutations under the action of position-shifting $\pi \mapsto \pi'$ with $\pi'_i = \pi_{(i \mod n) + 1}$ for $\pi \in \sym_n$. Cyclic permutations are marked with a bar, and we write its one-line notation ending with the largest element. For instance, the cyclic permutation $\cyc{\pi} = \cyc{42135}$ is the equivalent class $\{42135, 21354, 13542, 35421, 54213 \}$. We denote by $\cyc{\sym_{n+1}}$ the set of cyclic permutations with $n+1$ elements.

The pinnacles of a cyclic permutation $\cyc{\pi}$ is defined in the same way as for normal permutations, up to cyclic index, and we also define $\cyc{\sym_{n+1}(P)}$ analogously. We observe that $(n+1)$ is always a pinnacle of $\cyc{\pi} \in \cyc{\sym_{n+1}}$. We denote by $\cycf(\pi)$ the \tdef{cyclic completion} of $\pi \in \sym_n$, which is in $\cyc{\sym_{n+1}}$ and whose one-line notation is obtained by adding $(n+1)$ at the end of that of $\pi$. For instance, $\cycf(2413) = \cyc{24135}$. It is clear that $\cycf$ is a bijection between $\sym_n$ and $\cyc{\sym_{n+1}}$. Furthermore, it also preserves $\pin(\pi)$ in the following sense.

\begin{lem}[Lemma~4.2 in \cite{pin-sagan}] \label{lem:cyc}
  Given $\pi \in \sym_n$, we have $\pin(\cycf(\pi)) = \pin(\pi) \cup \{ n+1 \}$. In other words, $\cycf$ is a bijection from $\sym_n(P)$ to $\cyc{\sym_{n+1}(P \cup \{ n+1 \})}$.
\end{lem}

Given $\cyc{\pi} \in \cyc{\sym_{n+1}}$ and $2 \leq \ell \leq n+1$, we define the \tdef{segment set} $S_\ell(\cyc\pi)$ of level $\ell$ as the list of maximal consecutive segments (in the cyclic sense) in $\cyc{\pi}$ whose elements are at least $\ell$. We denote by $s_\ell(\cyc{\pi}) = |S_\ell(\cyc{\pi})|$. For instance, $S_4(\cyc{435281679})$ consists of the segments $5, 8, 6794$ with all elements at least $4$, and we thus have $s_4(\cyc{435281679}) = 3$. It is clear that $s_{n+1}(\cyc{\pi}) = s_2(\cyc{\pi}) = 1$.

We now define the \tdef{Motzkin type} of $\cyc{\pi}$, denoted by $\motzf(\cyc{\pi})$, as the lattice path of length $n-1$ whose $i$-th step is $(1, s_{n+1-i}(\cyc{\pi}) - s_{n+2-i}(\cyc{\pi}))$, for $i \in [n-1]$. In other words, $\motzf(\cyc{\pi})$ is the lattice path through the points $(i, s_{n+1-i}(\cyc{\pi}) - 1)$ for $0 \leq i \leq n-1$. Figure~\ref{fig:motzkin-type} shows a cyclic permutation and its Motzkin type, along with all $S_\ell(\cyc{\pi})$.

\begin{figure}[!thbp]
  \centering
  \includegraphics[page=1,width=0.6\textwidth]{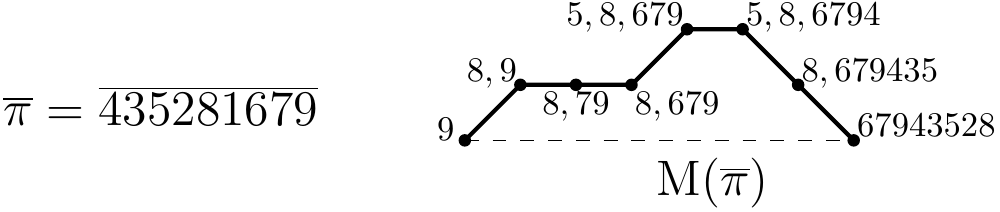}
  \caption{Example of a cyclic permutation, its Motzkin type and segment sets}
  \label{fig:motzkin-type}
\end{figure}

We recall that a \tdef{Motzkin path} of length $n$ is a lattice path starting at $(0,0)$, ending at $(n,0)$, composed by up steps $(1,1)$, horizontal steps $(1,0)$ and down steps $(1,-1)$, while always staying weakly above the $x$-axis. We denote by $\motz_{n-1}$ the set of Motzkin paths of length $n-1$. The name ``Motzkin type'' comes from the following lemma.

\begin{lem}\label{lem:motzkin-type}
  For $\cyc\pi \in \cyc{\sym_{n+1}}$, the lattice path $\motzf(\cyc{\pi})$ is a Motzkin path. Furthermore, the $i$-th step of $\motzf(\cyc{\pi})$ is an up step if and only if $n+1-i$ is a pinnacle of $\cyc{\pi}$.
\end{lem}
\begin{proof}
  We only need to show that $s_j(\cyc\pi) - s_{j+1}(\cyc\pi) \in \{-1, 0, 1\}$ for $j \geq 2$. The element $j$ in $\cyc\pi$ can only be in three cases:
  \begin{itemize}
  \item $j$ is not next to any element at least $j+1$, in this case $s_j(\cyc\pi) - s_{j+1}(\cyc\pi) = 1$, as we have the new segment $j$;
  \item $j$ is next to only one element $e$ at least $j+1$, in this case $s_j(\cyc\pi) - s_{j+1}(\cyc\pi) = 0$, as $j$ joins the segment of $e$;
  \item $j$ is next to two elements $e_1, e_2$ at least $j+1$, in this case $s_j(\cyc\pi) - s_{j+1}(\cyc\pi) = -1$. It is because $e_1$ and $e_2$ come from two different segments, as $j \geq 2$, and then $j$ fusions these two segments.
  \end{itemize}
  We thus conclude that $\motzf(\cyc{\pi})$ is a Motzkin path. For the second part, we observe that an up step means the creation of a new segment, which is equivalent to the new element being a pinnacle.
\end{proof}

Given a Motzkin path $M$, we define a weight on its steps. Up steps have weight $1$, horizontal steps on height $\ell$ have weight $2(\ell+1)$, and down steps from height $\ell$ to $\ell-1$ have weight $\ell(\ell+1)$. The weight of $M$, denoted by $w_\motz(M)$, is the product of the weights of its steps.

\begin{prop} \label{prop:sum-type}
  Given a Motzkin path $M$ of length $n-1$, the number of cyclic permutations $\cyc{\pi} \in \cyc{\sym_{n+1}}$ with $\motzf(\cyc{\pi}) = M$ is $w_\motz(M)$.
\end{prop}
\begin{proof}
  We construct $\cyc\pi$ with $\motzf(\cyc\pi) = M$ step by step. We start from $S_{n+1} = \{(n+1)\}$, then we follow $M$ step by step, and construct the segment set $S_{n+1-i}$ by adding the element $n+1-i$ to $S_{n+2-i}$. There are three possible cases.
  \begin{itemize}
  \item The $i$-th step of $M$ is an up step. Thus, $S_{n+1-i}$ has one more segment than $S_{n+2-i}$. The only way to do is to add $n+1-i$ as a new segment, that is, $S_{n+1-i} = S_{n+2-i} \cup \{ (n+1-i) \}$.
  \item The $i$-th step of $M$ is a horizontal step of height $h$, meaning that $n+1-i$ joins one of the $h+1$ segments. There are $2(h+1)$ such possibilities, as $n+1-i$ can be attached to both ends of a segment.
  \item The $i$-th step of $M$ is a down step from height $h$ to $h-1$, meaning that $n+1-i$ joins two of the $h+1$ segments. There are $h(h+1)$ possibilities, $(h+1)$ for the left segment and $h$ for the right one.
  \end{itemize}
  To see that each sequence of possible choices gives a unique permutation $\pi$, we observe that $\pi$ is totally determined by all the pairs $(\pi(j), \pi(j+1))$. In the $i$-th step, we see that each choice determines uniquely the pairs containing $n+1-i$ as the smaller element, and the sequence of choices thus determines $\pi$ uniquely. Furthermore, each such permutation can be constructed in this way by looking at its segment sets. The number of choices for each step is exactly its weight. Thus, the number of permutations with Motzkin type $M$ is $w_\motz(M)$.
\end{proof}

Given a set $P \subseteq [n]$, we denote by $\motz_{n-1,P}$ the set of Motzkin paths $M$ of length $n-1$ such that, for each $p \in P$, the $(n+1-p)$-th step of $M$ is an up step, and these are the only up steps of $M$. We have the following corollary.

\begin{coro}\label{coro:motzkin-pin}
  For a set $P \subseteq [n]$, we have
  \[
    |\sym_{n}(P)| = \sum_{M \in \motz_{n-1,P}} w_\motz(M).
  \]
\end{coro}
\begin{proof}
  Let $P' = P \cup \{n+1\}$. From Lemma~\ref{lem:cyc}, we have $|\sym_{n}(P)| = |\cyc{\sym_{n+1, P'}}|$. We then conclude by Lemma~\ref{lem:motzkin-type} and the definition of $\motz_{n-1,P}$.
\end{proof}

As a corollary, we have a simple proof of the lower bound of $|\sym_n(P)|$ for $|P|=k$.

\begin{coro}[Proposition~3.13 in \cite{pin-2018}]\label{coro:lower-bound}
  Given $n \geq 2k+1$, for every $P \subseteq [n]$ with $|P|=k$ and $\sym_n(P) \neq \varnothing$, we have
  \[
    |\sym_n(P)| \geq 2^{n-k-1}.
  \]
  The lower bound is reached by $P_{\min} = \{3, 5, 7, \ldots, 2k+1\}$.
\end{coro}
\begin{proof}
  As $\sym_n(P) \neq \varnothing$, by Corollary~\ref{coro:motzkin-pin}, there is at least one Motzkin path $M \in \motz_{n-1, P}$. Now, there are $n-2k-1$ horizontal step and $k$ down steps in $M$, each contributing a factor at least $2$. We thus have $|\sym_n(P)| \geq w_\motz(M) \geq 2^{n-k-1}$. For $P_{\min}$, from the definition of Motzkin type, it is clear that the only possible Motzkin type for $\cyc{\pi} \in \cyc{\sym_{n+1}}$ with $\pin(\cyc{\pi}) = P_{\min} \cup \{n+1\}$ is the one starting with $2^{n-2k-1}$ horizontal steps, then ending with $k$ pairs of up-and-down steps. This Motzkin path has exactly the minimal weight.
\end{proof}

\begin{rmk}
  The construction in Proposition~\ref{prop:sum-type} can be seen as a variant of the Françon-Viennot bijection between permutations and weighted Motzkin paths, but we use a slightly different set of weights that is more adapted to our analysis afterwards. It was also formulated in terms of increasing binary trees in \cite{flajolet}, which also provides a general theory to express the generating function of weighted Motzkin paths by continued fractions. In the original bijection, reformulated in our terms here, segments are already positioned in the permutation upon its creation, instead of free-floating. This leads to a weight $\ell+1$ on up steps starting at height $\ell$, while the weight of down steps starting at height $\ell$ is reduced to $\ell+1$. 
\end{rmk}

\section{Motzkin path compression}\label{sec:dyck}

A \tdef{Dyck path} is a Motzkin path without horizontal steps. It is clear that it has the same number of up steps and down steps. We denote by $\dyck_n$ the set of Dyck paths of length $2n$. Given a Motzkin path $M$, we define its \tdef{Dyck compression}, denoted by $\dyckc(M)$, the Dyck path obtained by removing all horizontal steps in $M$. We define the \tdef{Dyck type} $\dycktype(\cyc{\pi})$ of a cyclic permutation $\cyc{\pi}$ to be $\dyckc(\motzf(\cyc{\pi}))$. The following is a direct consequence of Lemma~\ref{lem:motzkin-type}.

\begin{prop} \label{prop:dyck-len-pin}
  For $\pi \in \sym_n$, the length of $\dycktype(\cycf(\pi))$ is $2|\pin(\pi)|$.
\end{prop}
\begin{proof}
  We know from Lemma~\ref{lem:cyc}~and~\ref{lem:motzkin-type} that pinnacles in $\pin(\pi)$ are in bijection with up steps in $\motzf(\cycf(\pi))$, which are kept in $\dycktype(\cycf(\pi))$.
\end{proof}

Given $n$ and $P = \{ p_1, p_2, \ldots, p_k \} \subseteq [n]$, we take the convention $p_0 = n+1$ and $p_{k+1} = 1$. We define the \tdef{gap sequence} $(g_0, g_1, \ldots, g_k)$ of $(n,P)$ by taking $g_i = p_i - p_{i+1} - 1$ for all $i$. The name comes from the fact that $g_i$ is the number of steps between the $i$-th and the $(i+1)$-st up step of Motzkin paths in $\motz_{n-1,P}$.

We start by a computational lemma. For $k \geq 1$, we denote by $h_m(x_1, x_2, \ldots, x_k)$ the homogeneous symmetric function of order $m$ with $k$ variables:
\begin{equation}
  \label{eq:def-sum-power}
  h_m(x_1, \ldots, x_k) = \sum_{m_1 + m_2 + \cdots + m_k = m, m_i \geq 0} x_1^{m_1} x_2^{m_2} \cdots x_k^{m_k}.
\end{equation}
Readers are referred to \cite[Chapter~7]{stanley} for more on symmetric functions. We have the following computational lemma for $h_m$.

\begin{lem} \label{lem:sum-power}
  For $m \in \mathbf{N}$, we have
  \[
    h_m(x_1, \ldots, x_k) = \sum_{1 \leq i \leq k} \frac{x_i^{m+k-1}}{\prod_{j \neq i} (x_i - x_j)}.
  \]
\end{lem}
\begin{proof}
  We proceed by induction on $k$. The equality holds trivially for $k = 1$. Now, to pass from $k$ to $k+1$, first write $h_m(x_1, \ldots, x_{k+1})$ as a polynomial in $x_{k+1}$ with coefficients $h_{m'}(x_1, \ldots, x_k)$, then apply the induction hypothesis on these coefficients. Comparing the resulting expression of $h_m(x_1, \ldots, x_{k+1})$ with the claimed one, the induction step is reduced to the proof of the following equality:
  \[
    \sum_{1 \leq i \leq k} x_i^{k-1} \cdot \frac{\displaystyle\prod_{1 \leq j \leq k, j \neq i} (x_{k+1} - x_j)}{\displaystyle\prod_{1 \leq j \leq k, j \neq i} (x_i - x_j)} = x_{k+1}^{k-1}.
  \]
  This equality holds by Lagrange interpolation.
\end{proof}

Now, given Proposition~\ref{prop:dyck-len-pin}, for each pair $(P, D)$ with $P \subseteq [n]$ with $|P| = k$, and $D$ a Dyck path of length $2k$, we associate a weight $w_\dyck(n, P, D)$ defined as follows. For $1 \leq i \leq k - 1$, let $d_i$ be the number of down steps between the $i$-th and the $(i+1)$-st up steps of $D$, and we also take $d_0 = 0$ and $d_k$ the number of down steps at the end of $D$. We also suppose that, for $1 \leq i \leq k$, the $i$-th up step of $D$ goes from height $\ell_i-1$ to $\ell_i$, and we take $\ell_0 = 0$. We thus have two tuples $(d_0, \ldots, d_k)$ and $(\ell_0, \ldots, \ell_k)$ depending only on $D$. Let $(g_0, g_1, \ldots, g_k)$ be the gap sequence of $(n, P)$. We then define $w_\dyck(n,P,D)$ by
\begin{equation} \label{eq:dyck-weight} 
  w_\dyck(n, P, D) = \prod_{0 \leq i \leq k} f(d_i, \ell_i, g_i),
\end{equation}
where $f(d, \ell, g)$ is defined as follows: $f(d, \ell, g) = 0$ when $d > g$, otherwise
\begin{align}
  \begin{split} \label{eq:drun-weight}
    f(d, \ell, g) &= ([\ell=0] + \ell(\ell+1)) h_{g-d}(\ell+1, \ell, \ldots, \ell-d+1) \\
    &= ([\ell=0] + \ell(\ell+1)) \sum_{m=0}^{d} (-1)^m \frac{(\ell+1-m)^{g}}{m! (d-m)!}.
  \end{split}
\end{align}
Here, $[h=0]$ is the Iverson bracket for the condition $h=0$, taking the value $1$ when $h=0$, and $0$ otherwise. We remark that $d_i = \ell_i - \ell_{i+1} + 1$, but we choose the current notation for simplicity. The last equality of \eqref{eq:drun-weight} is from Lemma~\ref{lem:sum-power}.

\begin{prop}\label{prop:dyck-weight}
  Given $P = \{ p_1 > p_2 > \cdots > p_k \} \in [n]$ and a Dyck path $D$ of length $2k$, the number of permutations $\pi$ in $\sym_n(P)$ with Dyck type $D$ is $w_\dyck(n, P, D)$.
\end{prop}
\begin{proof}
  Let $(d_0, \ldots, d_k)$ and $(\ell_0, \ldots, \ell_k)$ be the two tuples used in \eqref{eq:dyck-weight}, and $(g_0, \ldots, g_k)$ the gap sequence of $P$. The length of $D$ given $P$ is imposed by Proposition~\ref{prop:dyck-len-pin}. By Lemma~\ref{lem:cyc} and Proposition~\ref{prop:sum-type}, we only need to show that
  \[
    w_\dyck(n, P, D) = \sum_{M \in \motz_{n-1,P}, \; \dyckc(M) = D} w_\motz(M).
  \]
  By the definition of $\motz_{n-1,P}$, for each $M \in \motz_{n-1,P}$, the $(n+1-p_i)$-th step is an up step for each $1 \leq i \leq k$. We now consider the steps between the $i$-th and the $(i+1)$-st up step in $M$. There are $g_i$ such steps, $d_i$ of them are down steps, so $g_i-d_i$ of them are horizontal. These horizontal steps come with heights from $\ell_i$ to $\ell_i - d_i$. The contribution of all possibilities is
  \[
    \sum_{m_1+\cdots+m_{d_i+1} = g_i-d_i} 2^{g_i-d_i} \prod_{j=1}^{d_i+1} (\ell_i-j+1)^{m_j} = 2^{g_i-d_i} h_{g_i-d_i}(\ell_i+1, \ell_i, \ldots, \ell_i-d_i+1).
  \]
  For steps after the last up step in $M$, the same formula applies. It also holds for the $n-p_1 = g_0$ horizontal steps that comes before the first up step in $M$ by taking $d_0 = 0$. For down steps, we transfer their contribution to up steps of the same height in a Dyck path, as there are the same number of them. As the heights of the ending point of up steps are exactly those $\ell_i$ except for $\ell_0=0$, the contribution is $([\ell_i=0] + \ell_i(\ell_i+1))$ for all $i$. Collecting all factors, we have our claim by noticing that $\sum_{i=0}^k (g_i-d_i) = n - 1 - 2k$.
\end{proof}

\begin{rmk}
  In the definition of $w_\dyck(n, P, D)$, it may sound more natural if $f(d_i, h_i, g_i)$ accounts directly the weights of the $d_i$ down steps. However, the current form is simpler and faster to compute, and more adapted to the proof of Theorem~\ref{thm:q}.
\end{rmk}

\begin{thm}\label{thm:dyck-weight}
  For $P \in [n]$, let $k = |P|$, and we have
  \[
    |\sym_{n}(P)| = \sum_{D \in \dyck_k} w_\dyck(n, P, D).
  \]
\end{thm}
\begin{proof}
  This is a direct consequence of Proposition~\ref{prop:dyck-weight} by summing up all $D \in \dyck_k$.
\end{proof}

Our model here is conceptually simpler, and involves only positive terms (considering the functions $h_m$ as counting combinatorial objects), in contrast to the approach in \cite{pin-sagan} using the principle of inclusion-exclusion.

\section{Recurrence for $|\sym_n(P)|$} \label{sec:rec}

A \tdef{valley Dyck prefix} is a prefix of a Dyck path such that the rest of the path does not start with a down step, but can be empty. We can extend the definition of $w_\dyck(n,P,D)$ to $D$ being a valley Dyck prefix. Suppose that there are $k'$ up steps in $D$, then we take $w_\dyck(n,P,D) = \prod_{i=0}^{k'} f(d_i, \ell_i, g_i)$, with $d_i, \ell_i, g_i$ defined as in the case of Dyck paths, except $d_{k'}$ being the number of down steps at the end of $D$. Let $c_{n,P}(i,j)$ be the total weight $w_\dyck(n,P,D)$ for valley Dyck prefixes with $i$ up steps and $j$ down steps. Then we have the following recurrence.

\begin{thm}\label{thm:weight-rec}
  Given $n$ and $P = \{p_1 > p_2 > \cdots > p_k\} \subseteq [n]$ with $k = |P|$, let $(g_0, \ldots, g_k)$ the gap sequence of $P$. For $i < j$, we take $c_{n,P}(i,j) = 0$. For $0 \leq j \leq i \leq k$, we have
  \begin{align}
    c_{n,P}(0, 0) &= f(0, 0, g_0) = 1, \label{eq:rec-init} \\
    c_{n,P}(i+1, j) &= \sum_{j' = 0}^{j} f(j - j', i - j' + 1, g_{j+1}) c_{n,P}(i, j'). \label{eq:rec}
  \end{align}
  The function $f(d, h, g)$ is defined in \eqref{eq:drun-weight}. As a consequence, we have
  \[
    |\sym_n(P)| = 2^{n - 1 - 2k} \sum_{D \in \dyck_k}w_\dyck(n, P, D) = 2^{n - 1 - 2k} c_{n,P}(k, k).
  \]
\end{thm}
\begin{proof}
  The initial condition \eqref{eq:rec-init} stands for the empty valley Dyck prefix, and holds by definition. For the recurrence \eqref{eq:rec}, consider a valley Dyck prefix $D$ with $i+1$ up steps and $j$ down steps, and $D'$ the valley Dyck prefix of $D$ before the last up step of $D$. It is clear that $D'$ has $i$ up steps, and let $j' \leq j$ be the number of down steps in $D'$. There are thus $j-j'$ consecutive down steps at the end of $D$, starting at height $i-j'+1$. Their contribution to the weight is $f(j-j', i-j'+1, g_{j+1})$. By summing over all possibilities of $j'$, we have the recurrence. It is clear that $c_{n,P}(k,k)$ accounts for all Dyck paths, and we have the expression of $|\sym_n(P)|$ by Theorem~\ref{thm:dyck-weight}.
\end{proof}

The recurrence in Theorem~\ref{thm:weight-rec} provides an efficient way to compute $|\sym_n(P)|$. Now we describe a possible implementation as a proof of Proposition~\ref{prop:algo}.

\begin{proof}[of Proposition~\ref{prop:algo}]
  Suppose that $P \subseteq [n]$ with $k = |P|$. To compute $|\sym_n(P)|$, we use the recurrence in Theorem~\ref{thm:weight-rec}, which has $O(k^2)$ terms, and each term needs $O(k)$ multiplications and evaluations of $f(d,\ell,g)$, and each $f(d,\ell,g)$ is essentially a sum of $O(k)$ terms. We thus need $O(k^4)$ operations to sum up these terms. We now see that these $O(k^4)$ terms in $f(d,\ell,g)$ have the form $(\ell+1-m)^g(m!(d-m)!)^{-1}$, with $0 \leq m \leq d$. Moreover, there are only $k+1$ values of $g$ that are used, precisely those in the gap sequence, and we also have $1 \leq \ell+1-m \leq k+1$. Thus, we can precompute all such powers with $O(k^2\log n)$ operations using fast exponentiation, as $g_j \leq n$. The factorials can also be precomputed with $O(k)$ operations. With such precomputation, each of the $O(k^4)$ terms can be computed with a constant number of operations. The total arithmetic complexity is thus $O(k^4 + k^2 \log n)$.
\end{proof}

However, even without the optimizations in the proof of Proposition~\ref{prop:algo}, the recurrence in Theorem~\ref{thm:weight-rec} is fast enough for most practical use. For instance, for
\[
  P = \{ 97, 94, 85, 79, 68, 67, 63, 48, 43, 38, 25, 24, 23, 18, 13, 8, 3 \}, n = 100,
\]
a naïve implementation of the recurrence using Python gives the following result for $|\sym_n(P)|$ in 18 milliseconds on a low-end laptop with an 1.6GHz Intel Core i5-8250U: \numprint{2056053437 7719527577 7666916692 7111145600 8071023389 3827186696 7172893700 9544359423 5099087423 4585088000}. \; We also note that the algorithm works correctly even when $|\sym_n(P)| = 0$.

The recurrence in Theorem~\ref{thm:weight-rec} can also be easily computed symbolically using a computer algebra system, giving a closed-form, non-recursive formula for $|\sym_n(P)|$ for each fixed size of $P$. This settles Question~4.4 in \cite{pin-2018} for any fixed number of pinnacles. The cases $|P|=1, 2$ have already appeared in Proposition~3.6 of \cite{pin-2018} and reproved in \cite{pin-marne} in a much simpler way.

\begin{rmk}
  Another use of Theorem~\ref{thm:dyck-weight} is the exhaustive generation of all permutations in $\sym_n(P)$, which can be done in linear amortized time, \textit{i.e.}, each object generated in linear time on average. We first use the recurrence to identify viable Dyck paths. Then we decompress each Dyck path into a family of Motzkin paths by generating subsets of $d$ down steps in the gap $g$ for each contribution $f(d,\ell,g)$. This can be done by constant amortized time subset generators (see Chapter~7.2.1.3 of \cite{taocp}). For each Motzkin path, we generate choices for each step in the construction in the proof of Proposition~\ref{prop:sum-type}, which takes constant amortized time, and then perform the construction, which takes linear time. It might be possible to improve the complexity to constant amortized time, potentially by creating suitable data structure to efficiently transpose changes of the choices for each step into the permutation.
\end{rmk}

\section{A simple form of a weighted sum of $|\sym_n(P)|$} \label{sec:q}

We will now prove our formula of $q_n(P)$ in Theorem~\ref{thm:q}. It is equivalent to Conjecture~5.1 in \cite{pin-marne}, which will not be stated here (see Remark~\ref{rmk:q-conj}). We proceed as follows: we first propose a recurrence similar to that in Theorem~\ref{thm:weight-rec} and prove that it computes $q_n(P)$, then we prove inductively a general form for all terms in the recurrence, which implies our claim.

\begin{prop}\label{prop:q-rec}
  For $n \geq 1$ and $P = {p_1 > p_2 > \cdots > p_k} \subseteq [n]$, we take the convention that $p_0 = n+1$ and $p_{k+1} = 1$. We define $a_{n,P}(i,j)$ for $0 \leq i,j \leq k$ as follows. When $i < j$, we have $a_{n,P}(i,j) = 0$, and otherwise the following well-founded recurrence is satisfied:
  \begin{align}
    a_{n,P}(0,0) &= f(0, 0, p_0 - p_1 - 1) = 1 \label{eq:q-init} \\
    \begin{split}\label{eq:q-rec}
      a_{n,P}(i+1,j) &= [i+1=j] f(0, 0, p_0 - p_{i+2} - 1) \\
      &\quad+ \frac1{2} \sum_{i'=0}^{i} \sum_{j'=0}^{j-i+i'} f(i'-j'-i+j, i'-j'+1, p_{i'+1} - p_{i+2} - 1) a_{n,P}(i',j').
    \end{split}
  \end{align}
  Here, $[i+1=j]$ is the Iverson bracket, which takes $1$ when $i+1=j$, and $0$ otherwise. Then,
  \[
    q_n(P) = 2^{n-1} a_{n,P}(k,k).
  \]
\end{prop}
\begin{proof}
  Given $M$ a Motzkin path, we denote by $|M|_{\mathrm{up}}$ the number of up steps in $M$. By Lemma~\ref{lem:motzkin-type} and Proposition~\ref{prop:sum-type}, $q_n(P)$ is the weighted sum of Motzkin paths $M$ of length $n-1$ whose up steps can occur as the $(n+1-p_i)$-th step for all $1 \leq i \leq k$, but not necessarily, and the weight is given by $2^{|M|_{\mathrm{up}}} w_\motz(M)$. We denote this set of Motzkin paths by $\widetilde\motz_{n-1,P}$, which is exactly $\bigcup_{P' \subseteq P} \motz_{n-1, P'}$, and we have
  \begin{equation} \label{eq:q-motz}
    q_n(P) = \sum_{M \in \widetilde\motz_{n-1,P}} 2^{|M|_{\mathrm{up}}} w_\motz(M).
  \end{equation}

  We now propose an altered weight $\widetilde{w}_\motz(M)$, where an up step has weight $1/2$, a horizontal step on height $\ell$ has weight $\ell+1$, and a down step from height $\ell$ to $\ell-1$ has weight $\ell(\ell+1)$. We observe that a Motzkin path $M$ of length $n-1$ has $n-1-2|M|_{\mathrm{up}}$ horizontal steps. Comparing step weights in $w_\motz(M)$ and $\widetilde{w}_\motz(M)$, we have
  \[
    \widetilde{w}_\motz(M) = 2^{1-n} \cdot 2^{|M_{\mathrm{up}}|} w_\motz(M).
  \]
  Combining with \eqref{eq:q-motz}, we have
  \[
    q_n(P) = 2^{n-1} \sum_{M \in \widetilde\motz_{n-1,P}} \widetilde{w}_\motz(M).
  \]

  The proof of the recurrence is analogous to that of Theorem~\ref{thm:weight-rec}. This time, we consider \tdef{valley Motzkin prefixes}, which are prefixes such that the rest of the path is empty or starts with an up step. The altered weight $\widetilde{w}_\motz(M)$ also extends to these prefixes. We now claim that $a_{n,P}(i,j)$ is the weighted sum of valley Motzkin prefixes of length $n-p_{i+1}$, ending on height $i-j$, with the weight given by $\widetilde{w}_\motz$. In other words, these are valley Motzkin prefixes of Motzkin paths in $\widetilde\motz_{n-1,P}$ whose $(n+1-p_{i+1})$-th step is an up step. We must have $0 \leq j \leq i$, as there may be at most $i$ up steps in the valley Motzkin prefixes accounted by $a_{n,P}(i,j)$ for any $i$. 

  We now prove our claim by induction on $i$. The initial case $i=0$ is given by definition. Now suppose that our claim holds for all $i' \leq i$. Consider a valley Motzkin prefix $M$ of length $n-p_{i+2}$ ending on height $i+1-j$, which is supposed to contribute to $a_{n,P}(i+1,j)$. Either $M$ has no up step, which can only occur when $i+1=j$, and in this case its contribution is $1$, the same as the Iverson bracket in \eqref{eq:q-rec}; or $M$ as at least one up step. Let $M'$ be the prefix of $M$ before its last up step, which is also a valley Motzkin prefix. Suppose that $M'$ is of length $n-p_{i'+1}$ and ending at height $i'-j'$ for some $i'$ and $j'$. It is clear that $0 \leq i' \leq i$, and we also have $i'-j' \geq i-j$, as there is exactly one up step in $M$ after $M'$. The segment of $M$ after its last up step starts at height $i'-j'+1$, contains $p_{i'+1}-p_{i+2}-1$ steps, $i'-j'-i+j$ of them are down steps by variation of height. Therefore, we have
  \[
    \widetilde{w}(M) = \frac1{2} \sum_{i'=0}^{i} \sum_{j'=0}^{j-i+i'} f(i'-j'-i+j, i'-j'+1, p_{i'+1} - p_{i+2} - 1) \widetilde{w}(M'),
  \]
  where the factor $1/2$ is for the last up step of $M$. Summing over all possible $i', j'$ and $M'$, we conclude the induction. We then conclude the proof by observing that $a_{n,P}(k,k)$ accounts for all Motzkin paths in $\widetilde{\motz}_{n-1,P}$.
\end{proof}

In the proof of Theorem~\ref{thm:q}, we need the following computational result.

\begin{prop}\label{prop:binom-poly}
  For $n > 0$ and $P(x)$ a polynomial of degree strictly less than $n$, we have
  \[
    \sum_{k=0}^n (-1)^k \binom{n}{k} P(k) = 0.
  \]
\end{prop}
\begin{proof}
  By linearity, it suffices to prove for $P(x)=x(x-1)\cdots(x-k+1)$ for all $k < n$. This is done by deriving $k$ times by $y$ the expansion of $(1-y)^n$, then setting $y=1$.
\end{proof}

We recall that $\mathcal{R}_{k,s}$ is the set of $y$-coordinate sequences of Dyck meanders of length $k$ ending at height $s$, and $\mathcal{R}_k = \cup_{s=0}^k \mathcal{R}_{k,s}$. Given $r \in \mathcal{R}_k$, we denote by $r^\uparrow$ (resp. $r^\downarrow$) the new sequence obtained by appending $r_{k+1} = r_k+1$ (resp. $r_{k+1} = r_k-1$). We can now prove Theorem~\ref{thm:q}.

\begin{proof}[of Theorem~\ref{thm:q}]
  We prove the following result for $a_{n,P}(i,j)$ for $0 \leq j \leq i \leq k$:
  \begin{equation}
    \label{eq:q-simple}
    a_{n,P}(i,j) = 2^{-i} \sum_{s=i-j}^{i} \frac{s!}{(s-i+j)!} \sum_{r \in \mathcal{R}_{i, s}} w_\meand(r).
  \end{equation}
  It implies our claim, as $q_n(P) = 2^{n-1} a_{n,P}(k, k)$.
  
  We proceed by induction on $i$. For $i = 0$, the sum in $a_{n,P}(0,j)$ is empty except for $j=0$, when we have $a_{n,P}(0,0) = 1$. Now assume that \eqref{eq:q-simple} holds for all $i' \leq i$, and we compute $a_{n,P}(i+1,j)$ using the induction hypothesis and \eqref{eq:q-rec}. We first check the case $i+1 \neq j$, where the Iverson bracket \eqref{eq:q-rec} is $0$, leading to
  \begin{align}
    a_{n,P}(i+1,j)
    &= 2^{-i-1} \sum_{i'=0}^{i} \sum_{j'=0}^{j-i+i'}  \sum_{s=i'-j'}^{i'} \frac{s!}{(s-i'+j')!} \sum_{r \in \mathcal{R}_{i', s}} w_\meand(r) \nonumber \\
    &\quad \cdot \left[ (i'-j'+1)(i'-j'+2) \sum_{t=0}^{i'-j'-i+j} (-1)^t \frac{(i'-j'+2-t)^{p_{i'+1} - p_{i+2} - 1}}{t!(i'-j'-i+j-t)!} \right] \nonumber \\
    &= 2^{-i-1} \sum_{i'=0}^{i} \sum_{s=i-j}^{i'} \sum_{r \in \mathcal{R}_{i', s}} w_\meand(r) \alpha(s, i-j, p_{i'+1} - p_{i+2} - 1), \label{eq:sum-alpha}
  \end{align}
  where
  \begin{align}
    \begin{split} \label{eq:alpha-def}
    \alpha(s, \delta, u) &= \sum_{\delta' = \delta}^{s} \frac{s!}{(s-\delta')!} \cdot \left[ (-1)^{\delta'} (\delta'+1)(\delta'+2) \sum_{t=\delta}^{\delta'} (-1)^t \frac{(t+2)^u}{(\delta'-t)!(t-\delta)!} \right] \\
    &= \sum_{t=\delta}^{s} (-1)^t \frac{(t+2)^u s!}{(s-t)!(t-\delta)!} \sum_{\delta' = t}^{s} (-1)^{\delta'} (\delta'+1)(\delta'+2) \binom{s-t}{\delta'-t}.
    \end{split}
  \end{align}
  Here, we used the generic expression without Iverson bracket of $f(d,\ell,g)$ in \eqref{eq:drun-weight}, as we never have $i'-j'+1 = 0$ in the case $i+1 \neq j$. We can also assume $\delta \geq 0$. For $s-t > 2$, by Proposition~\ref{prop:binom-poly}, the summation over $\delta'$ is zero. Therefore, when $s \geq \delta + 2$, the only non-zero terms are for $t=s-2, s-1, s$, leading to
  \begin{equation}\label{eq:alpha-3}
    \alpha(s, \delta, u) = \frac{(s+2)^{u+1} (s+1)!}{(s-\delta)!} - \frac{2(s+1)^{u+1} s!}{(s-\delta-1)!} + \frac{s^{u+1} (s-1)!}{(s-\delta-2)!}.
  \end{equation}
  However, when $s=\delta+1$, we discard the last term, and when $s=\delta$, only the first term exists.

  We now remark that we may extend the definition of $w_\meand$ to $r_-$ that may take $-1$ as the last element, and we have $w_\meand(r_-) = 0$ in such cases. Recalling that $g_{i+1} = p_{i+1}-p_{i+1}-1$, the contribution of $r \in \meand_{i',s}$ to $a_{n,P}(i+1,j)$ can thus be expressed by $C_+(r) - C_-(r)$ from \eqref{eq:alpha-3}, where, for $s \geq i-j+2$,
  \begin{align}
    \begin{split}\label{eq:contrib}
    C_+(r) &= 2^{-i'-1} w_\meand(r) \left[ \frac{(s+2)^{p_{i'+1}-p_{i+2}}(s+1)!}{(s-i+j)!} + \frac{s^{p_{i'+1}-p_{i+2}}(s-1)!}{(s-i+j-2)!} \right], \\
    C_-(r) &= 2^{-i'} w_\meand(r) \frac{(s+1)^{p_{i'+1}-p_{i+2}}s!}{(s-i+j-1)!}.
    \end{split}
  \end{align}
  For $s = i-j+1$, the last term in \eqref{eq:alpha-3} is discarded, and we only have the first term of $C_+(r)$, while $C_-(r)$ remains unchanged. For $s = i-j$, not only $C_+(r)$ has only one term, but the term in \eqref{eq:alpha-3} corresponding to $C_-(r)$ is also discarded, leaving $C_-(r) = 0$. For $s < i-j$, we have $C_+(r) = C_-(r) = 0$.
  
  We may also extend the definition of $C_+$ and $C_-$ to $r_-$ that may take $-1$ as the last element, and in this case, $C_+(r_-) = C_-(r_-) = 0$, as $w_\meand(r_-) = 0$. The value of $a_{n,P}(i+1,j)$ is the sum of all contributions:
  \[
    a_{n,P}(i+1,j) = \sum_{i'=0}^{i} \sum_{s=i-j}^{i'} \sum_{r \in \meand_{i',s}} (C_+(r) - C_-(r)).
  \]

  For $s \geq i-j+2$ and $r \in \meand_{i',s}$, we have $w_\meand(r^\uparrow) = w_\meand(r) (s+2)^{p_{i'+1} - p_{i'+2}}$ and $w_\meand(r^\downarrow) = w_\meand(r) s^{p_{i'+1} - p_{i'+2}}$. With this fact, we have $C_+(r) - C_-(r^\uparrow) - C_-(r^\downarrow) = 0$ by checking carefully that, in \eqref{eq:contrib}, the first term of $C_+(r)$ is exactly $C_-(r^\uparrow)$, and the second is $C_-(r^\downarrow)$. When $s = i-j+1$ or $s = i - j$, we only have the first term in $C_+(r)$, but in this case $C_-(r^\downarrow) = 0$, and the same still holds. %When $s=i-j$, we also have $C_-(r^\downarrow) = 0$, but $C_-(r^\uparrow)$ remains, and the same also holds.
  
  Therefore, summing over all paths, the only contribution terms $C_+(r)$ that is left is when $r \in \meand_{i,s}$ with $s \geq i-j$, and the only $C_-(r)$ left is for $r \in \meand_{i-j,i-j}$, and we have $C_-(r) = 0$ for such $r$. Using the fact that $C_-(r) = 0$ for $r \in \meand_{i+1,i-j}$, we have
  \begin{align*}
    a_{n,P}(i+1,j) &= \sum_{s=i-j}^{i} \sum_{r \in \meand_{i,s}} C_+(r) = \sum_{s=i-j+1}^{i+1} \sum_{r \in \meand_{i+1, s}} C_-(r) \\
                   &= 2^{-i-1} \sum_{s=i+1-j}^{i+1} \sum_{r \in \meand_{i+1,s}} w_\meand(r) \frac{s!}{(s-i+j-1)!}
  \end{align*}

  For the case $i+1 = j$, there are several differences that compensate each other. First, in this case, not all the $i'$ and $j'$ in the recurrence \eqref{eq:q-rec} are valid. More precisely, when $j'=i'+1$, which is possible when $i+1=j$, we have $a_{n,P}(i',j') = 0$. Therefore, we should discard the term linked with $j'=i'+1$, which also means that we can continue to use the generic expression without Iverson bracket of $f(d,\ell,g)$ in \eqref{eq:drun-weight}.

  For terms being discarded, firstly, we may only sum over $s$ from $0$ to $i'$ in \eqref{eq:sum-alpha}. Therefore, we always have $s \geq 0$ in $\alpha(s, \delta,u)$. Furthermore, in $\alpha(s,-1,u)$, we discard the terms with $\delta=-1$ in \eqref{eq:alpha-def}, but not those with $t = -1$. When $s \geq 1$, this has no effect. When $s=0$, we discard the last term on the right-hand side of \eqref{eq:alpha-3}. The effect on $C_+(r)$ and $C_-(r)$ is the same as the general case, and the whole reasoning holds, except that we didn't account for the negative contribution term $C_-((0)) = 1$. However, it compensates with the Iverson bracket in \eqref{eq:q-rec}, and the same result holds. We thus complete the induction on all cases.
\end{proof}

\begin{rmk}\label{rmk:q-conj}
  Conjecture~5.1 in \cite{pin-marne} involves a relatively simple but non-trivial iterative procedure to generate an expression of $q_n(P)$. However, it is equivalent to Theorem~\ref{thm:q}. We only briefly describe the reason here, without a detailed proof. In the procedure in \cite{pin-marne}, there are two operators $f_e$ and $f_o$, one used for even number of pinnacles, the other for the odd case. In our model of lattice paths, both operators ``extend'' the path by one step, up or down. The apparent difference between $f_e$ and $f_o$ is due to the fact that, when paths are extended to odd height (so even weight for horizontal steps), a power of $2$ is collected somewhere else. 

  Theorem~\ref{thm:q} also explains the phenomenon observed in \cite{pin-marne} that, in the computation of $q_n(P)$, we may have $2$ in $P$, although it cannot be a pinnacle. For $P \subseteq [n] \setminus \{1,2\}$, let $k = |P|$. It is clear that $q_n(P) = q_n(P \cup \{2\})$, as $\sym_n(P') = \varnothing$ for any $P'$ containing $2$. We check that, for each $r \in \mathcal{R}_{k}$, the contribution of $r^\uparrow$ and $r^\downarrow$ to $q_n(P \cup \{2\})$ is exactly twice that of $r$ to $q_n(P)$, which is compensated by the increment of number of pinnacles.
\end{rmk}

In \cite{pin-2018}, it was proposed that, as $q_n(P)$ may have a simpler form, it may be easier to compute $|\sym_n(P)|$ using $q_n(P)$ and the principle of inclusion-exclusion. However, the arithmetic complexity of such computation is exponential in $|P|$, in contrast to Proposition \ref{prop:algo} here.

Although the statement of Theorem~\ref{thm:q} has a strong combinatorial flavor, its proof here is mainly computational, which is not quite satisfying. We thus have the following natural question:
\begin{open}
  Is there a combinatorial or even bijective proof of Theorem~\ref{thm:q}?
\end{open}
If such a proof exists, it might lead us to a better understanding why the form of $q_n(P)$ is much simpler than that of $|\sym_n(P)|$.

\section{Counting pinnacle orders} \label{sec:pin-order}

Given $P = \{ p_1 > p_2 > \cdots > p_k \}$, for some $\pi \in \sym_n(P)$ for $n \geq p_1$, the \tdef{pinnacle ordering} of $\pi$, denoted by $\ord(\pi)$, is a permutation $\sigma \in \sym_k$ such that $p_1, \ldots, p_k$ appear in $\pi$ in the order $p_{\sigma(1)}, p_{\sigma(2)}, \ldots, p_{\sigma(k)}$. For instance, for the permutation $\pi = 46352817$, its pinnacle set is $\pin(\pi) = \{8, 6, 5\}$, and we have $\ord(\pi) = 231$, as the pinnacles in $\pin(\pi)$ appear in $\pi$ in the order $6, 5, 8$.

A permutation $\sigma \in \sym_k$ is an \tdef{admissible pinnacle ordering} (or simply \tdef{admissible ordering}) of $P$ if there is $\pi \in \sym_n(P)$ with $\ord(\pi) = \sigma$. We denote by $\ordset(P)$ the set of admissible pinnacle ordering of $P$. Note that $\ordset(P)$ does not depend on $n$, as in every $\pi \in \sym_n(P)$, any element greater than $p_1$ cannot be a peak, and they must thus be located at the two ends of $\pi$. Hence, taking any $n > p_1$ will not lead to more admissible orderings than taking $n = p_1$.

%As $\ord(\pi)$ is also a permutation, we may consider the segment set of $\cyc{\ord(\pi)}$ and its relation with that of $\cyc{\pi}$. Given a word $w$ in $[n]$ and a set $P \subseteq [n]$, we denote by $\Pi_P(w)$ the \tdef{projection} of $w$ on $P$, \textit{i.e.}, the word with all letters not in $P$ removed. By abuse of notation, for a set of words $S$, we denote by $\Pi_P(S)$ the set of projections of words in $S$.

Given $\sigma \in \sym_k$, we want to know whether $\sigma$ is in $\ordset(P)$. We have the following characterization of possible pinnacle orders $\sigma$ for $\cycf(\pi)$ of a given Dyck type, using the Motzkin type of $\cycf(\sigma)$. To simplify the proof, for $\sigma \in \sym_k$, we define its \tdef{complement}, denoted by $\hat\sigma$, by $\hat\sigma(i) = k + 1 - \sigma(i)$. We recall that, for a cyclic permutation $\cyc\pi$, its segment set of level $\ell$, denoted by $S_\ell(\cyc\pi)$, is the set of maximal consecutive segments formed by elements at least $\ell$ in $\cyc\pi$, and we define $s_\ell(\cyc\pi) = |S_\ell(\cyc\pi)|$.

\begin{prop} \label{prop:pin-ord-seg}
  Given $P \subseteq [n]$ with $|P| = k$, we fix $D$ a Dyck path of length $2k$ such that there is some $\pi_* \in \sym_n(P)$ with $\dycktype(\cycf(\pi_*)) = D$. For a permutation $\sigma \in \sym_k$, we let $M$ be the Motzkin type of $\cycf(\sigma)$, that is, $M = \motzf(\cycf(\sigma))$. We have the following equivalence:
  \begin{enumerate}
  \item[(i)] There exists $\pi \in \sym_n(P)$ with $\dycktype(\cycf(\pi)) = D$ and $\ord(\pi) = \hat\sigma$;
  \item[(ii)] The Motzkin path $M$ is \tdef{compatible} with the Dyck path $D$, that is, for all $1 \leq i \leq k-1$, the starting height of the $i$-th step of $M$ never exceeds that of the $i$-th \emph{up step} of $D$.
  \end{enumerate}
\end{prop}
\begin{proof}
  For simplicity, we take $\cyc{\pi} = \cycf(\pi)$ and $\cyc{\sigma} = \cycf(\sigma)$, and we suppose that $P = \{ p_1 > p_2 > \cdots > p_k \}$. By construction, the pinnacles of $\cyc\pi$ are exactly $n+1, p_1, p_2, \ldots, p_{i-1}$, corresponding to elements $k+1, k, \ldots, k-i+2$ in $\cyc\sigma$, as $\ord(\pi) = \hat\sigma$ and $n+1$ is always a pinnacle in $\cyc\pi$. By the definition of $\motzf$, the starting height of the $i$-th step of $M$ is $s_{k+2-i}(\cyc{\sigma}) - 1$. Similarly, from Lemma~\ref{lem:motzkin-type} and the definition of $\dycktype$, we know that the $i$-th up step of $D$ corresponds to the pinnacle $p_i$, and its starting height is $s_{p_i+1} (\cyc{\pi}) - 1$. Hence, in the following, we may replace (ii) by
  \begin{enumerate}
  \item[(ii)'] For all $1 \leq i \leq k-1$, we have $s_{k+2-i}(\cyc{\sigma}) \leq s_{p_i+1}(\cyc{\pi})$.
  \end{enumerate}

  We start by proving (i) $\Rightarrow$ (ii)'. The elements in each segment in $S_{k+2-i}(\cyc\sigma)$ correspond to pinnacles in $\cyc\pi$. However, if two elements $j, j'$ are not in the same segment of $S_{k+2-i}(\cyc\sigma)$, then there are two elements $m, m' < k+2-i$ that separate $j$ and $j'$ in $\cyc\sigma$, which transposes to two pinnacles $p_{k-m+1}, p_{k-m'+1} \leq p_i$ that separate $p_{k-j+1}$ and $p_{k-j'+1}$ in $\cyc\pi$. It means that $p_{k-j+1}$ and $p_{k-j'+1}$ are not in the same segment in $S_{p_i+1}(\cyc{\pi})$. The restriction of segments in $S_{p_i+1}(\cyc{\pi})$ to pinnacles is thus a refinement of $S_{k+2-i}(\cyc\sigma)$ as set partition, and we have $s_{k+2-i}(\cyc{\sigma}) \leq s_{p_i+1}(\cyc{\pi})$ to conclude (i) $\Rightarrow$ (ii)'.

  To show (ii) $\Rightarrow$ (i), we only need to show a way to construct $\pi$ satisfying (i). Knowing that there is $\pi_* \in \sym_n(P)$ such that $\dycktype(\cycf(\pi_*)) = D$, as in the proof of Proposition~\ref{prop:sum-type}, we may alter any choices in the procedure to construct $\cycf(\pi_*)$ from its Motzkin type $\motzf(\cycf(\pi_*))$ to construct $\cyc\pi$ of the same Motzkin type, thus also with the same Dyck type. We leave choices for horizontal steps unchanged, focusing only on down steps. We now read each step of $M$ and its choice in the construction of $\cyc\sigma$ from $M$, and translate them to choices of down steps in $\motzf(\cycf(\pi_*))$, which are just down steps in $D$. With an abuse of language that will be justified later, given an element or a segment in $\cyc\sigma$, we say that its corresponding segment in $\cyc\pi$ is the one containing the corresponding pinnacles. There are three possible steps in $M$:
  \begin{itemize}
  \item Up step: do nothing;
  \item Horizontal step: suppose that it joins a new element $i$ to an existing segment $w$ in $S_{i+1}(\cyc\sigma)$, then the next down step in $D$ should join the corresponding segments of $i$ and of $w$ in the same way;
  \item Down steps: suppose that it joins a new element $i$ with two segments $w$ and $w'$, forming $wiw'$, then we take the next \emph{two} down steps in $D$ to join the corresponding segments of $i$, $w$ and $w'$ in the same way, with arbitrary order.
  \end{itemize}
  The procedure above is well-defined, as we mimic in $\cyc\pi$ how segments join in $\cyc\sigma$, meaning that the corresponding pinnacles of elements in the same segment of $\cyc\sigma$, once it comes to existence during the reading of $M$, are also in the same segment of $\cyc\pi$ after the translated steps. The only way that the process above fails is when it is about to deal with a new element in $\cyc\sigma$ for the $i$-th step of $M$, but the corresponding pinnacle $p_i$ is not yet introduced by the $i$-th up step of $\motzf(\cyc\pi)$ (thus also that of $D$). Suppose that the starting height of the $i$-th step of $M$ is $\ell_i$ and that of the $i$-th up step of $D$ is $\ell_i'$. Steps in $M$ for elements larger than $i$ correspond by the process above to $i - 1 - \ell_i$ down steps in $D$, as a step in $M$ with height increment $1$, $0$, $-1$ consumes $0$, $1$, $2$ down steps in $D$ respectively. Then, for the $i$-th up step of $D$, there are $i-1-\ell_i'$ down steps before it. Therefore, our process fails for the $i$-th step of $M$ when the first $i-1$ steps in $M$ produces strictly less down steps than those before the $i$-th up step in $D$, that is, when $i - 1 - \ell_i < i - 1 - \ell_i'$. But this is impossible due to (ii). Therefore, our process never fails, and it produces a permutation $\pi$ with $\ord(\pi) = \hat\sigma$.
\end{proof}

Given a pinnacle set $P$, a Dyck path $D$ is a \tdef{admissible Dyck type} of $P$ if there is some $\pi \in \sym_n(P)$ such that $\dycktype(\cycf(\pi)) = D$. By Proposition~\ref{prop:pin-ord-seg}, to obtain $|\ordset(P)|$, we only need to compute all Motzkin paths of length $k-1$ that is compatible with any of the admissible Dyck types. The following characterization of admissible Dyck types of $P$ allows us to focus on only one of them, instead of all, when computing $|\ordset(P)|$.

\begin{prop} \label{prop:dyck-type-cond}
  Let $P = \{ p_1 > p_2 > \cdots > p_k \} \subseteq [n]$ and $D$ an admissible Dyck type of $P$. Let $\ell_i$ be the starting height of the $i$-th up step of $D$. Then for all $2 \leq i \leq k$, we have
  \[
    \ell_i \leq \min\left(\ell_{i-1} + 1, p_i - 3 - 2(k - i) \right).
  \]
\end{prop}
\begin{proof}
  It is clear that $\ell_i \leq \ell_{i-1} + 1$, as there is no up step between the $(i-1)$-st and the $i$-th up step. Now for the other inequality, as $D$ is admissible, there is some $\pi \in \sym_n(P)$ such that $\dycktype(\cycf(\pi)) = D$. Consider all the steps of $D$ after the $i$-th up step (itself included), among them $k-i+1$ are up steps, and thus $\ell_i + k - i + 1$ are down steps, as $D$ is a Dyck path. By the definition of $\dycktype$ and $\motzf$, these steps are from distinct elements from $p_i$ down to $2$. Hence, we must have more elements than down steps, meaning that $\ell_i + 2(k-i+1) \leq p_i - 1$, thus $\ell_i \leq p_i - 3 - 2(k-i)$.
\end{proof}

As a simple corollary, we recover the characterization of admissible pinnacle sets (\textit{i.e.}, pinnacle sets $P$ such that $\sym_n(P)$ is not empty) in \cite{pin-2018}, here stated and strengthened for our need.

\begin{coro}[See Proposition~2.3 in \cite{pin-2018}] \label{coro:admissible-pin}
  Let $P = \{ p_1 > p_2 > \cdots > p_k \} \subseteq [n]$. Then $\sym_n(P)$ is not empty if and only if $p_i \geq 3 + 2(k-i)$ for all $2 \leq i \leq k$.

  Furthermore, when this condition is satisfied, there is an admissible Dyck type $D_P$ for $P$, with all $\ell_i$ reaching the maximum. $D_P$ can be constructed as follows: for each $2 \leq i \leq k$, we put $\max(0, \ell_{i-1} - p_i + 2 + 2(k-i))$ down steps between the $(i-1)$-st up step and the $i$-th one, and we add the correct number of down steps at the end to make $D_P$ a Dyck path. For $D_P$, we have $\ell_i = \min\left(\ell_{i-1} + 1, p_i - 3 - 2(k - i) \right)$.
\end{coro}
\begin{proof}
  For the ``only if'' part, suppose that there is an index $i$ with $p_i < 3 + 2(k-i)$. If $\sym_n(P)$ contains some $\pi$, we take $D = \dycktype(\cycf(\pi))$, and by Proposition~\ref{prop:dyck-type-cond}, the starting height $\ell_i$ of the $i$-th step of $D$ satisfies $\ell_i \leq p_i - 3 - 2(k-i) < 0$, which is impossible. Therefore, such an index $i$ cannot exist.

  For the ``if'' part, we first show that $D_P$ constructed above is an admissible Dyck type. As $p_i \geq 3+2(k-i)$, the number of down steps between the $(i-1)$-st and the $i$-th up steps is at most $\ell_{i-1} + 1$, and accounting for the $(i-1)$-st up step, we have $\ell_i \geq 0$, showing that $D_P$ is indeed a Dyck path. To show that $D_P$ is admissible, we expand $D_P$ into a Motzkin path $M$ that satisfies the condition in Lemma~\ref{lem:motzkin-type}. This is always possible, as $p_i \geq 3+2(k-i)$ implies that there are enough elements for down steps at each time, as in the proof of Proposition~\ref{prop:dyck-type-cond}. With $M$ a valid Motzkin type, by Proposition~\ref{prop:sum-type}, we see that $\sym_n(P)$ is not empty.

  The equality for $\ell_i$ holds by construction. Now, let $D'$ be an admissible Dyck type of $P$, with $\ell'_i$ the starting height of its $i$-th step. We show by induction that $\ell'_i \leq \ell_i$. This holds for $i=1$ as $\ell'_i = \ell_i = 0$. Suppose that $\ell'_{i-1} \leq \ell_{i-1}$, then by Proposition~\ref{prop:dyck-type-cond},
  \[
    \ell'_i \leq \min\left(\ell'_{i-1} + 1, p_i - 3 - 2(k - i) \right) \leq \min\left(\ell_{i-1} + 1, p_i - 3 - 2(k - i) \right) = \ell_i.
  \]
  We thus conclude the maximality of $D_P$.
\end{proof}

Given $P \subseteq [n]$ with $\sym_n(P)$ not empty, we call the Dyck path $D_P$ defined in Corollary~\ref{coro:admissible-pin} to be the \tdef{maximal Dyck type} of $P$. We can now characterize permutations in $\ordset(P)$ in a simple way.

\begin{thm} \label{thm:ordset-chara}
  Given $P \subseteq [n]$ with $|P| = k$. A permutation $\sigma$ is in $\ordset(P)$ if and only if the Motzkin type of $\cycf(\hat\sigma)$ is compatible (as defined in Proposition~\ref{prop:pin-ord-seg}) with the maximal Dyck type $D_P$ of $P$. We thus have
  \[
    |\ordset(P)| = \sum_{\substack{M \in \motz_{k-1} \\ M \text{ compatible with } D_P}} w_\motz(M).
  \]
\end{thm}
\begin{proof}
  The ``if'' part follows directly from Proposition~\ref{prop:pin-ord-seg} by taking $D = D_P$ therein, as $\sigma \mapsto \hat\sigma$ is involutive. For the ``only if'' part, as $\sigma \in \ordset(P)$, there exists $\pi \in \sym_n(P)$ such that $\ord(\pi) = \sigma$. By Proposition~\ref{prop:pin-ord-seg}, the Motzkin type of $\cycf(\hat\sigma)$ is compatible with the Dyck type of $\cycf(\pi)$, meaning that it is also compatible with $D_P$, according to Corollary~\ref{coro:admissible-pin} and the definition of compatibility. By accounting for all possible $M$, using Proposition~\ref{prop:sum-type}, we have the formula of $|\ordset(P)|$.
\end{proof}

Following \cite{rusu-tenner}, we say that a pinnacle set $P$ with $k = |P|$ is \tdef{maximally admissible} if $\ordset(P) = \sym_k$. The following corollary recovers a result in \cite{rusu-tenner}.

\begin{coro}[Corollary~4.6 in \cite{rusu-tenner}]
  A pinnacle set $P = \{ p_1 > p_2 > \cdots > p_k \} \subseteq [n]$ is maximally admissible if and only if $p_i \geq \min(2k - i + 2, 3(k+1-i))$ for all $2 \leq i \leq k - 1$.
\end{coro}
\begin{proof}
  From Theorem~\ref{thm:ordset-chara}, we know that $P$ is maximally admissible if and only if $D_P$ is compatible with all possible Motzkin paths of length $k-1$. By definition, it is equivalent to $D_P$ being compatible with the Motzkin path that starts with $\lfloor \frac{k-1}{2} \rfloor$ up steps and ends with $\lfloor \frac{k-1}{2} \rfloor$ down steps. This is equivalent to say that $\ell_i \geq \min(i - 1, k - i)$ for $\ell_i$ for $D_P$ defined in Corollary~\ref{coro:admissible-pin}.

  For the ``only if'' part, it is clear that, for $2 \leq i \leq k - 1$, we have $p_i - 3 - 2(k-i) \geq \min(i - 1, k - i)$, which implies what we want. For the ``if'' part, we proceed by induction on $i$. The case $i = 1$ is trivial, as we have $\ell_1 = 0 \geq \min(0, k)$. Suppose that $\ell_{i-1} \geq \min(i - 2, k - i + 1)$, then by Corollary~\ref{coro:admissible-pin} and the condition on $p_i$, we have
  \[
    \ell_i = \min(\ell_{i-1} + 1, p_i - 3 - 2(k-i)) \geq \min(\min(i-1, k-i+2), \min(i-1, k-i)).
  \]
  We thus have $\ell_i \geq \min(i-1, k-i)$ to conclude the induction.
\end{proof}

Using Theorem~\ref{thm:ordset-chara}, we propose the following recurrence for $|\ordset(P)|$ that avoids the enumeration of the exponentially many Motzkin paths.

\begin{prop} \label{prop:ordset-rec}
  Given $P = \{ p_1 > p_2 > \cdots > p_k \} \subseteq [n]$, we define $\ell_1 = 0$ and $\ell_i = \min(\ell_{i-1} + 1, p_i - 3 - 2(k-i))$ for $2 \leq i \leq k$. We then define $b_P(i,j)$ for $-1 \leq i, j \leq k$ as follows. When $j > \ell_{i-1}$ or $j < 0$, we have $b_P(i, j) = 0$. Otherwise, we have
  \begin{align*}
    b_P(0, 0) &= 1, \\
    b_P(i + 1, j) &= b_P(i, j - 1) + 2(j+1)b_P(i, j) + (j+1)(j+2)b_P(i, j + 1).
  \end{align*}
  Then we have $|\ordset(P)| = b_P(k-1, 0)$.
\end{prop}
\begin{proof}
  The proof is similar to that of Theorem~\ref{thm:weight-rec} and Proposition~\ref{prop:q-rec}. We first extend the weight $w_\motz$ to prefixes of Motzkin paths. Then, the quantity $b_P(i, j)$ stands for the total weight of prefixes of Motzkin paths compatible with $D_P$ ending at $(i, j)$, where $D_P$ is defined in Corollary~\ref{coro:admissible-pin}. At $b_P(k-1,0)$, we thus have the total weights of Motzkin paths compatible with $D_P$, and we conclude by Theorem~\ref{thm:ordset-chara}.
\end{proof}

The recurrence in Proposition~\ref{prop:ordset-rec} gives a concrete algorithm for Proposition~\ref{prop:pin-order-algo}.
  
\begin{proof}[of Proposition~\ref{prop:pin-order-algo}]
  Let $k = |P|$. Computing $|\ordset(P)|$ using the recurrence in Proposition~\ref{prop:ordset-rec} needs two steps: computing $\ell_i$ for $i \in [k]$, which takes $O(k)$ operations; computing $b_P(i,j)$ for $i, j \in [k]$, which takes $O(k^2)$ operations.
\end{proof}

Using the algorithm above, we are able to compute $|\ordset(P)|$ for most practical purposes. We observe that, for fixed $k = |P|$, there is a finite number of possible non-zero values for $|\ordset(P)|$. We denote this number by $\alpha_k$. A natural bound of $\alpha_k$ is the $(k-1)$-st Motzkin number, as the restriction imposed by $D_P$ essentially means summing over the weights of all Motzkin paths of length $k-1$ below a given ``ceiling'', which is also a Motzkin path, when computing $|\ordset(P)|$. However, this bound is not tight. For instance, a Motzkin path and its mirror, as ceilings, lead to the same weighted sum. An exhaustive computation shows that, for $k$ from $1$ to $19$, the values of $\alpha_k$ are:
\[
  1, 1, 2, 3, 6, 10, 21, 38, 86, 173, 412, 926, 2331, 5713, 14981, 38750, 104907, 279344, 769429
\]
At the time of writing, this sequence is not yet on the \textit{Online Encyclopedia of Integer Sequences}~\cite{oeis}. It may thus be interesting to study $(\alpha_k)_{k \geq 1}$. A precise description of $\alpha_k$ may be complicated, as it involves various symmetries and coincidences due to explicit values of the weights. However, we may expect $(\alpha_k)_{k \geq 1}$ to have the same asymptotic behavior as Motzkin numbers.

\begin{open}
  Does $(\alpha_k)_{k \geq 1}$ grow exponentially with the growth constant $3$, the same as that of Motzkin numbers?
\end{open}

For a Sagemath worksheet containing simple implementations of the algorithms in this article, see \url{https://igm.univ-mlv.fr/~wfang/code/Pinnacle-Code.ipynb}.

\bibliographystyle{abbrvnat}
\bibliography{pinnacle}

\end{document}